\documentclass[12pt]{amsart}
\usepackage{amssymb,latexsym, amsmath, amsxtra}
\usepackage{graphicx}
\newtheorem{conjecture}{Conjecture}
\newtheorem{theorem}{Theorem}

\begin{document}
\title{Averages of ratios of the Riemann zeta-function and correlations of divisor sums}
 
\author{Brian Conrey}
\address{American Institute of Mathematics, 600 East Brokaw Rd., San Jose, CA 95112, USA and School of Mathematics, University of Bristol, Bristol BS8 1TW, UK}
\email{conrey@aimath.org}
\author{Jonathan P. Keating}
\address{School of Mathematics, University of Bristol, Bristol BS8 1TW, UK}
\email{j.p.keating@bristol.ac.uk}

\thanks{We gratefully acknowledge support under EPSRC Programme Grant EP/K034383/1
LMF: L-Functions and Modular Forms.  Research of the first author was also supported 
by the American Institute of Mathematics and by a grant from the National Science 
Foundation. JPK is grateful for the following additional support: a grant from the 
Leverhulme Trust, a Royal Society Wolfson Research 
Merit Award, and a Royal Society Leverhulme Senior Research Fellowship. He is also pleased to thank the American Institute of Mathematics for hospitality.}

\date{\today}

\begin{abstract} We establish a connection between the  ratios conjecture   for the Riemann zeta-function and a conjecture concerning correlations of convolutions of M\"{o}bius and divisor functions.  Specifically, we prove that the ratios conjecture and an arithmetic correlations conjecture imply the same result.  This provides new support for the ratios conjecture, which previously had been motivated by analogy with formulae in random matrix theory and by a heuristic recipe.  Our main theorem generalises a recent calculation pertaining to the special case of two-over-two ratios.
\end{abstract}

\maketitle

\section{Introduction and statement of results}  

Let $A$ and $B$ be sets of complex numbers with real parts smaller than $1/4$.
Let $C$ and $D$    be sets of complex numbers with positive real parts smaller than $1/4$.
The purpose of this paper is to  investigate the averages 
$$\mathcal R_{A,B,C,D}(T):=\int_0^\infty \psi\left(\frac t T\right) \frac{\prod_{\alpha\in A}\zeta(s+\alpha)\prod_{\beta\in B}\zeta(1-s+\beta)}{\prod_{\gamma\in C}\zeta(s+\gamma)
\prod_{\delta\in D}\zeta(1-s+\delta)}~dt$$
where $s=1/2+it$. $\mathcal R$ is the subject of the ``ratios conjecture'' originally formulated 
in [CFZ] and studied in [CS]. In these prior studies the perspective was from the point of view of analogy with Random Matrix Theory (RMT).
Our new perspective is to study this quantity from an arithmetic point of view. In particular, we identify those parts of the ratios conjecture that arise from a study of the coefficient correlations 
$$\sum_{n\le X} I_{A,C}(n) I_{B,D}(n+h)$$
where  $I_{A,C}$ is defined implicitly by 
$$\sum_{n=1}^\infty \frac{I_{A,C}(n)}{n^s}=\frac{\prod_{\alpha\in A}\zeta(s+\alpha)}{\prod_{\gamma\in C} \zeta(s+\gamma)}
. 
$$
In this paper we will describe this connection explicitly.

 Not surprisingly,  $\mathcal R$  is related to averages of the (analytic continuation of the) Rankin-Selberg convolution
\begin{eqnarray*}
\mathcal{B}_{A,B,C,D}(s):=\sum_{n=1}^\infty \frac{I_{A,C}(n)
I_{B,D}(n)}{n^s}.
\end{eqnarray*}
In fact, we can state the ratios conjecture in a relatively simple way in terms of $\mathcal B$.
\begin{conjecture} ([CFZ] and [CS]) 
Suppose that the sets $A,B,C$ and $D$ are as in the introduction and that the imaginary 
parts of all of the parameters in this set are $O(T^{1-\xi})$ for some $\xi>0$. Then 
\begin{eqnarray*} \label{eqn:rat}
\mathcal R_{A,B,C,D}(T)=\int_0^\infty  \psi\left(\frac t T\right) 
\sum_{U\subset A, V\subset B\atop
|U|=|V|} \left(\frac{t}{2\pi}\right)^{-\sum_{\hat \alpha\in U}\hat \alpha-\sum_{\hat \beta \in V}
\hat \beta}
\mathcal{B}_{A-U+V^-,B-V+U^-,C,D}(1)~dt+O(T^{1-\eta})
 \end{eqnarray*}
 for some $\eta>0$.  
 \end{conjecture}
 Here $V^-$ denotes the set obtained from $V$ by replacing every element by its negative.
 So, the set $A-U+V^-$ may be obtained from the set $A$ by deleting the elements of $U$ and then 
 inserting the negatives of the elements of $V$.  (We assume that the elements of all of these sets are distinct. The situation with repeating elements may be deduced from this case by a limiting argument.)
 
  It is also not surprising that $\mathcal R$ is connected  to 
 weighted averages over $n$ and $h$ of
 $$  I_{A,C}(n)I_{B,D}(n+h).$$
It is this connection that we are elucidating.

Using  the $\delta$-method [DFI] it transpires that the  weighted averages relevant to the consideration of $\mathcal R$ may 
be expressed in terms of 
 \begin{eqnarray*}&&
 \mathcal C_{A,B,C,D}(s):=\frac{1}{(2\pi i)^2}\int_{|w-1|=\epsilon}
 \int_{|z-1|=\epsilon}\chi(w+z-s-1)
 \sum_{q=1}^\infty 
 \sum_{h=1}^\infty \frac{r_q(h)}{h^{s+2-w-z}} \\
 &&\qquad  \qquad \qquad \qquad \times 
  \sum_{m=1}^\infty \frac{I_{A,C}(m)e(m/q)}{m^w}
 \sum_{n=1}^\infty \frac{I_{B,D}(n)e(n/q)}{n^z}
 ~dw ~dz
  \end{eqnarray*}
  where $r_q(h)$ denotes Ramanujan's sum and where $\chi(s)$ is the factor 
  from the functional equation
  $\zeta(s)=\chi(s)\zeta(1-s)$; also here and elsewhere $\epsilon $ is chosen to be 
  larger than the absolute values of the shift parameters
  $\alpha,\beta,\gamma,\delta$ but smaller than $1/2$.
 
 The main conclusion of this paper is that the arithmetic contributions arising from the 
 averages of $I_{A,C}(n)I_{B,D}(n+h)$ coincide exactly with the terms from the ratios conjecture
 with $|U|=|V|=1$, i.e. what are referred to elsewhere as the ``one-swap'' terms.
    
  The  result that explicates this  is encapsulated in the following identity.
    \begin{theorem}
  Assuming the Generalized Riemann Hypothesis
  \begin{eqnarray*}
  \mathcal C_{A,B,C,D}(s)=\sum_{U\subset A, V\subset B\atop
  |U|=|V|=1}
  \mathcal{B}_{A-U+V^-,B-V+U^-,C,D}(s+1).
  \end{eqnarray*}
  \end{theorem}

  It turns out to be convenient  to study an average of the ratios conjecture. To this end
let
$$\mathcal I_{A,C}(s;X)= \sum_{n\le X}
I_{A,C}(n)n^{-s}.
$$
We are interested in the average over $t$ of $\mathcal I_{A,C}(s,X)\overline{\mathcal I_{B,D}}(1-s,X)$   (N.B.~$s=1/2+it$) in the   case  that $X=T^\lambda$ for some $\lambda>1$ .
(When $\lambda<1$ this average is dominated by  diagonal terms.)
We give two different  treatments of the average of ``truncated'' ratios:
$$
\mathcal M_{A,B,C,D}(T;X):=\int_0^\infty \psi\left(\frac t T\right)  \mathcal I_{A,C}(s,X)
  \mathcal I_{B,D}(1-s,X)~dt$$
 (where again $s=1/2+it$) which lead to the same answer. The first is by the ratios conjecture and 
  the second is by consideration of the correlations of the coefficients.

  In each case we prove 
  \begin{theorem} Let $A,B,C,D$ be as above.
  Then, assuming either a uniform version of the ratios conjecture 
  or a uniform version of the conjectural formula for correlations of values of $I_{\alpha,\gamma}(n)$, we have for some $\eta>0$ and some $\lambda>1$,
  \begin{eqnarray*} &&
  \mathcal M_{ A,B,C,D}(T;X)=\\
  &&   \quad 
  \int_0^\infty  \psi\left(\frac t T\right) \frac{1}{2\pi i}\int_{\Re s=2}
  \sum_{U\subset A, V\subset B\atop
|U|=|V|\le 1} \left(\frac{t}{2\pi }\right)^{-|U|s-\sum_{\hat \alpha\in U}\hat \alpha-\sum_{\hat \beta \in V}
\hat \beta}
\mathcal{B}_{A-U+V^-,B-V+U^-,C,D}(s+1)\frac{X^s}{s}~ds~dt\\
&&\qquad \qquad +O(T^{1-\eta}).
  \end{eqnarray*}
   \end{theorem}
  
  This shows that the ratios conjecture follows not only from the `recipe' of [CFZ], but also relates to  correlations of values of $I_{A,C}(n)$.
  
  An earlier paper [CK5] had this calculation but in the special case that all of the sets $A,B,C,D$ are singletons. That paper has additional background information and motivation, in particular relating to connections with correlations between the zeros (c.f.~[BK1, BK2]) and with the moments (c.f.~[CK1, CK2, CK3, CK4]) of the zeta function.  The first part of the present calculation follows closely that in [CK5].
  
  \section{Approach via the ratios conjecture}
  
  We have
  $$\mathcal I_{A,C}(s,X)=\frac{1}{2\pi i} \int_{(2)} 
  \mathcal I_{A,C}(s+w)\frac{X^w}{w}~dw;
  $$
  there is a similar expression for $\mathcal I_{B,D}(s,X)$. 
  Inserting these expressions and rearranging the integrations we have 
  \begin{eqnarray*}
  \mathcal M_{A,B,C,D}(T;X)=\frac{1}{(2\pi i)^2}
  \int_{\Re w=2}\int_{\Re z=2} \frac{X^{w+z}}{wz} \mathcal R_{
A_w,B_z,C_w,D_z}(T)  ~dw ~dz.
\end{eqnarray*}
We note that Conjecture 1 implies that 
$\mathcal R_{A_w,B_z,C_w,D_z}$
 is, to leading order as $T\rightarrow\infty$, a function of $z+w$.  We therefore make the 
 change of variable $s=z+w$; now the integration in the $s$ variable is
 on the vertical line $\Re s=4$. We retain $z$ as our
 other variable and integrate over it. This turns out to be the integral
 $$\frac{1}{2\pi i} \int_{\Re z = 2}\frac{dz}{z(s-z)}=\frac{1}{s}$$
  as is seen by moving the path of integration to the left to $\Re z=-\infty$.
Thus we have that $\mathcal M_{A,B,C,D}(T;X)$ is given to leading order by
\begin{eqnarray*}
 \frac{1}{2\pi i}
  \int_{\Re s=4} \frac{X^{s}}{s} \mathcal R_{
A_s,B,C_s,D}(T)  ~ds.
\end{eqnarray*}
Moving the path of integration to $\Re s=\epsilon$, avoiding any poles, inserting Conjecture 1, and noting that
$$\mathcal B_{A_s,B,C_s,D}(1)=\mathcal B_{A,B,C,D}(s+1),$$
we have that the uniform ratios conjecture implies the conclusion of Theorem 2.

  \section{Approach via coefficient correlations}
  We follow the approach developed by Goldston and Gonek [GG] in their work  
  on mean-values of long Dirichlet polynomials.

Expanding the sums and integrating term-by-term, we have
$$\mathcal M_{\alpha,\beta,\gamma,\delta}(T;X)=T\sum_{m,n\le X} 
\frac{I_{A,C}(m)I_{B,D}(n)}{\sqrt{mn}}\hat\psi
\left(\frac{T}{2\pi}\log\frac mn\right).
$$
\subsection{Diagonal} The diagonal term is 
$$T\hat{\psi}(0) \sum_{m\le X} 
 \frac{I_{A,C}(m)I_{B,D}(m)}{m}.
 $$
 By Perron's formula this sum is 
 \begin{eqnarray*}
 \frac{1}{2\pi i}\int_{(2)} \mathcal B_{A,B,C,D}(s+1)
 \frac{X^s}{s} ~ds.
 \end{eqnarray*}
 
\subsection{Off-diagonal}
 For the off-diagonal terms we  need to analyze
 $$2T\sum_{T\le m \le X}\sum_{1\le h \le \frac{X}{T}} \frac{I_{A,C}(m)
 I_{B,D}(m+h)}{m}\hat\psi
\left(\frac{Th}{2\pi m } \right).
$$
We replace the arithmetic terms by their average and express this as 
 $$2T\int_T^X \sum_{1\le h \le \frac{X}{T}} \frac{\langle I_{A,C}(m)I_{B,D}(m+h)\rangle _{m\sim u}
 }{u}\hat\psi
\left(\frac{Th}{2\pi u } \right) ~du.
$$
We now compute the average  heuristically via the delta-method [DFI]:
$$\langle I_{A,C}(m)I_{B,D}(m+h)\rangle _{m\sim u}
\sim \sum_{q=1}^\infty r_q(h) \langle I_{A,C}(m) e(m/q) \rangle_{m\sim u}
\langle I_{B,D}(m) e(m/q)  \rangle_{m\sim u}
$$
where $r_q(h)$ is the Ramanujan sum, a formula for which is $r_q(h)=
\sum_{d\mid h\atop d\mid q} d\mu(\frac q d )$.  This may be formalized as  a precise conjecture exactly as in Section 5 of [CK5].  It is this conjectural formula that we refer to in Theorem 2.
Now 
\begin{eqnarray*}
\langle I_{A,C}(m) e(m/q) \rangle_{m\sim u}
=\frac{1}{2\pi i}\int_{|w-1|=\epsilon}
\sum_{m=1}^\infty I_{A,C}(m)e(m/q) m^{-w} u^{w-1}~dw.
\end{eqnarray*}

The off-diagonal contribution is thus 
\begin{eqnarray*}&&
2T \sum_{1\le h \le \frac{X}{T}} \int_T^X
\frac{1}{(2\pi i)^2} \iint_{|w-1|=\epsilon\atop |z-1|=\epsilon} \sum_{q=1}^\infty r_q(h)
 \hat\psi
\left(\frac{Th}{2\pi u } \right)  u^{w+z-2} 
\\
&&\qquad \qquad \times \sum_{m_1=1}^\infty \frac{I_{A,C}(m_1)e(m_1/q)}{m_1^w}
\sum_{m_2=1}^\infty \frac{I_{B,D}(m_2)e(m_2/q)}{m_2^z}
~dw ~dz
~\frac{du}{u}.
\end{eqnarray*}
We next make the change of variables $v=\frac{Th}{2\pi u }$. The inequality $u\le X$ then implies that 
$\frac{Th}{2\pi v} \le X$ or $h\le \frac{2\pi v X}{T}$.   The above can be re-expressed as
\begin{eqnarray*}&&
2T\int_0^\infty \sum_{1\le h \le \frac{2\pi vX}{T}} 
\frac{1}{(2\pi i)^2} \iint_{|w-1|=\epsilon\atop |z-1|=\epsilon} \sum_{q=1}^\infty r_q(h)
 \hat\psi(v) \left(\frac{Th}{2\pi v}\right)^{w+z-2} 
\\
&&\qquad \qquad \times \sum_{m_1=1}^\infty \frac{I_{A,C}(m_1)e(m_1/q)}{m_1^w}
\sum_{m_2=1}^\infty \frac{I_{B,D}(m_2)e(m_2/q)}{m_2^z}
~dw ~dz
~\frac{dv}{v}.
\end{eqnarray*}
Using Perron's formula to express the sum over $h$ gives
\begin{eqnarray*}&&
2T\int_0^\infty  
\frac{1}{(2\pi i)^3}\int_{\Re s=2} \iint_{|w-1|=\epsilon\atop |z-1|=\epsilon} \sum_{q=1}^\infty 
\sum_{h=1}^\infty \frac{r_q(h)}{h^s}
 \hat\psi(v) \left(\frac{Th}{2\pi v}\right)^{w+z-2} \left(\frac {2\pi v X}T\right)^{s} 
\\
&&\qquad \qquad \times \sum_{m_1=1}^\infty \frac{I_{A,C}(m_1)e(m_1/q)}{m_1^w}
\sum_{m_2=1}^\infty \frac{I_{B,D}(m_2)e(m_2/q)}{m_2^z}\frac{ds}{s}
~dw ~dz
~\frac{dv}{v}.
\end{eqnarray*}
Now 
\begin{eqnarray*}
2\int_0^\infty \hat{\psi}(v) v^A \frac{dv}{v}=\chi(1-A)\int_0^\infty \psi(t) t^{-A} ~dt.
\end{eqnarray*}
Incorporating this formula gives 
\begin{eqnarray*}&&
T\int_0^\infty  \psi(t) 
\frac{1}{(2\pi i)^3}\int_{\Re s=2} \iint_{|w-1|=\epsilon\atop |z-1|=\epsilon} \sum_{q=1}^\infty 
\sum_{h=1}^\infty \frac{r_q(h)}{h^{s+2-w-z}}
  \left(\frac{Tt}{2\pi }\right)^{w+z-2} \left(\frac {2\pi X}{tT}\right)^{s} \chi(w+z-s-1)
\\
&&\qquad \qquad \times \sum_{m_1=1}^\infty \frac{I_{A,C}(m_1)e(m_1/q)}{m_1^w}
\sum_{m_2=1}^\infty \frac{I_{B,D}(m_2)e(m_2/q)}{m_2^z}\frac{ds}{s}
~dw ~dz
~dt.
\end{eqnarray*}
By Theorem 1, this is 
\begin{eqnarray*}&&\sum_{\hat \alpha\in A\atop \hat \beta\in B}
\int_0^\infty  \psi\left(\frac t T\right) 
\frac{1}{2\pi i}\int_{\Re s=2}  \left(\frac{t}{2\pi }\right)^{-\hat \alpha-\hat \beta-s}
 \mathcal B_{A'\cup\{-\hat \beta\},B'\cup \{-\hat \alpha\},C,D}(s+1)\frac{X^s}{s}~ds
~dt
\end{eqnarray*}
where $A'=A-\{\hat \alpha\}$ and $B'=B-\{\hat \beta\}$.
 Adding the diagonal and off-diagonal terms, 
 we thus obtain that the  conjecture for the correlations of values of $I_{A,C}(n)$ also implies the
 conclusion of Theorem 2.

\section{Proof of Theorem 1}

First of all, we have
$$\sum_{h=1}^\infty \frac{r_q(h)}{h^A} =\sum_{h=1}^\infty \frac{ \sum_{g\mid q\atop g\mid h}g
\mu(\frac q g)}{h^A}=\sum_{g\mid q} g^{1-A} \mu(\frac q g ) \zeta(A)
=q^{1-A}\Phi(1-A,q)\zeta(A)$$
where
$$\Phi(x,q)=\prod_{p\mid q}\left(1-\frac{1}{p^x}\right).$$
Using this and the functional equation for $\zeta$, we have to evaluate
\begin{eqnarray*}&&
\frac{1}{(2\pi i)^2} \iint_{|w-1|=\epsilon\atop |z-1|=\epsilon} \sum_{q=1}^\infty 
q^{w+z-s-1}\Phi(w+z-s-1,q)
   \\
&&\qquad \qquad \times
\zeta(w+z-s-1)
 \sum_{m_1=1}^\infty \frac{I_{A,C}(m_1)e(m_1/q)}{m_1^w}
\sum_{m_2=1}^\infty \frac{I_{B,D}(m_2)e(m_2/q)}{m_2^z} 
~dw ~dz.
\end{eqnarray*}

We identify the polar structure of the Dirichlet series here by passing to characters
via the formula
$$e\left(\frac m q\right) = \sum_{d\mid m\atop d\mid q}
\frac{1}{\phi\left(\frac q d\right)}\sum_{\chi \bmod \frac q d } \tau(\overline{\chi})
\chi\left(\frac m d \right).$$
Assuming GRH, the only poles near $w=1$ arise from the principal characters 
$\chi_{\frac qd}^{(0)}$. Using 
$$\tau(\chi_{\frac qd}^{(0)})=\mu(\frac qd)$$
we have that  the poles of $\sum_{m=1}^\infty 
I_{A,C}(m)e(m/q) m^{-w}$ are the same as the poles of 
\begin{eqnarray*} &&
\sum_{d\mid q}\frac{\mu\left(\frac q d\right)}{\phi\left(\frac q d\right)}
 \sum_{m=1}^\infty I_{A,C}(md) \chi_{\frac qd}^{(0)}(m) m^{-w}d^{-w}\\
 &&\qquad = q^{-w}
 \sum_{d\mid q}\frac{\mu (d)}{\phi(d)}d^{w}
 \sum_{m=1}^\infty \frac{I_{A,C}(\frac{mq}d) \chi_{d}^{(0)}(m)}{ m^{w}} 
 \end{eqnarray*}
 and the principal parts are the same.
 We now replace $\chi_d^{(0)}(m)$ by $\sum_{e\mid d\atop e\mid m}\mu(e)$.
 This leads to
  \begin{eqnarray*}
 q^{-w}  \sum_{d\mid q}\frac{\mu (d)d^w}{\phi(d)}\sum_{e\mid d}\mu(e)e^{-w}
 \sum_{m=1}^\infty \frac{I_{A,C}(\frac{meq}d)  }{m^{w}} .
 \end{eqnarray*}
Now we need the polar structure of 
$$
 \sum_{m=1}^\infty I_{A,C}(mr)  m^{-w} 
 $$
 for $r=qe/d$.
 
 Since $I_{A,C}(n)$ is a multiplicative function of $n$, $I_{A,C}(nr)/I_{A,C}(r)$ is also a multiplicative function of $n$. The generating function  may therefore be expressed as an Euler product:
 \begin{eqnarray*}
 \sum_{n=1}^\infty 
 \frac{I_{A,C}(nr)/I_{A,C}(r)}{n^w}= \sum_{n=1}^\infty 
 \frac{I_{A,C}(n)}{n^w} 
 \prod_{p\mid r}
 \frac{
    \sum_{j=0}^\infty    \frac{I_{A,C}(p^{j+\lambda_r(p)})
    /I_{A,C}(p^{\lambda_p(r)}}{p^{jw}}   }
    {\sum_{j=0}^\infty 
     \frac{I_{A,C}(p^{j  })
    }{p^{jw}} }
 \end{eqnarray*}

 This gives 
 \begin{eqnarray*}
 \sum_{n=1}^\infty 
 \frac{I_{A,C}(nr)}{n^w}&=& 
 \frac{\prod_{\alpha\in A}\zeta(w+\alpha)}{\prod_{\gamma\in C}\zeta(w+\gamma)} \prod_{p\mid r}
 \frac{
    \sum_{j=0}^\infty    \frac{I_{A,C}(p^{j+\lambda_r(p)})
   }{p^{jw}}   }
    {\sum_{j=0}^\infty 
     \frac{I_{A,C}(p^{j  })
    }{p^{jw}} }\\
    &=&  \frac{\prod_{\alpha\in A}\zeta(w+\alpha)}{\prod_{\gamma\in C}\zeta(w+\gamma)} 
    E_{A,C}(w,r),
 \end{eqnarray*}
 say. 
In particular, the poles are at $w=1-\alpha$ for $\alpha\in A$.
Thus, the integral over $w$ and $z$ is 
\begin{eqnarray*}&&
 \sum_{\hat \alpha\in A\atop
\hat  \beta\in B} \sum_{q=1}^\infty 
q^{1-\hat \alpha-\hat \beta-s}\Phi(1-\hat \alpha-\hat \beta-s,q)\zeta(1-\hat \alpha-\hat \beta-s)
   \\
&&\qquad \qquad \times
 q^{-1+\hat\alpha}  \sum_{d_1\mid q}\frac{\mu (d_1)d_1^{1-\hat\alpha}}{\phi(d_1)}\sum_{e_1\mid d_1}\mu(e_1)e_1^{-1+\hat\alpha }
  q^{-1+\hat\beta}  \sum_{d_2\mid q}\frac{\mu (d_2)d_2^{1-\hat\beta}}{\phi(d_2)}\sum_{e_2\mid d_2}\mu(e_2)e_2^{-1+\hat\beta }
 \\&&
 \qquad \qquad \times
 \frac{\prod_{\alpha\in A'}\zeta(1-\hat\alpha +\alpha)\prod_{\beta\in B'}\zeta(1-\hat\beta+\beta)}{\prod_{\gamma\in C}\zeta(1-\hat \alpha +\gamma)\prod_{\delta\in D}\zeta(1-\hat \beta +\delta)} 
    E_{A,C}(1-\hat \alpha,\frac{qe_1}{d_1})
      E_{B,D}(1-\hat \beta,\frac{qe_2}{d_2}).
\end{eqnarray*}

 So we have to identify the Dirichlet series
 \begin{eqnarray*}&&
 \sum_{q=1}^\infty q^{-1-s}\Phi(1-\hat \alpha-\hat \beta-s,q)\mathcal E_{A,C}(1-\hat\alpha, q)
 \mathcal E_{B,D}(1-\hat\beta, q)\\&&\qquad =
 \sum_{r=1}^\infty \frac{\mu(r)}{r^{2-\hat\alpha-\hat\beta}}\sum_{q=1}^\infty \frac{
 \mathcal E_{A,C}(1-\hat\alpha, qr)
 \mathcal E_{B,D}(1-\hat\beta, qr)}{q^{1+s}}
 \end{eqnarray*}
 where we have made use of $\Phi(\xi ,q)=\sum_{r\mid q} \mu(r) r^{-\xi}$, and where 
 \begin{eqnarray*}
 \mathcal E_{A,C}(1-\hat\alpha, q)=
  \sum_{d\mid q}\frac{\mu (d)d^{1-\hat\alpha}}{\phi(d)}\sum_{e\mid d}\mu(e)e^{-1+\hat\alpha }
  E_{A,C}(1-\hat \alpha,\frac{qe}{d}).
   \end{eqnarray*}

 This itself may be expressed as an Euler product. So, let's assume $q=p^J$ with $J\ge 1$ and identify
 \begin{eqnarray*}&&
   \sum_{d\mid q}\frac{\mu (d)d^{1-\hat\alpha}}{\phi(d)}\sum_{e\mid d}\mu(e)e^{-1+\hat\alpha }
 E_{A,C}(1-\hat \alpha,\frac{qe}{d})\\&&
 = E_{A,C}(1-\hat \alpha,p^J)-\frac{p^{1-\hat \alpha}}{p-1}E_{A,C}(1-\hat \alpha,p^{J-1})
 +\frac{1}{p-1}E_{A,C}(1-\hat \alpha,p^{J})\\&&
 =  \frac{p}{p-1}E_{A,C}(1-\hat \alpha,p^J)-\frac{p^{1-\hat \alpha}}{p-1}E_{A,C}(1-\hat \alpha,p^{J-1}).
 \end{eqnarray*}
 Now we note the identity
 $$I_{A,C}(p^J)=I_{A',C}(p^J)+p^{-\alpha}I_{A,C}(p^{J-1})$$
 where $A=A'\cup \{\alpha\}$.
 Thus
 \begin{eqnarray*}
     \sum_{j=0}^\infty    \frac{I_{A,C}(p^{j+J})
   }{p^{jw}}  -p^{-\alpha}   \sum_{j=0}^\infty    \frac{I_{A,C}(p^{j+J-1})
   }{p^{jw}}
   = \sum_{j=0}^\infty    \frac{I_{A',C}(p^{j+J})
   }{p^{jw}} .
 \end{eqnarray*}
 Thus,
 \begin{eqnarray*}
 \mathcal E_{A,C}(1-\hat\alpha, p^J)&=& \frac{p}{p-1}
 \frac{
  \sum_{j=0}^\infty \frac{I_{A',C}(p^{j+J})
   }{p^{j(1-\hat \alpha)}}  }
   {  \sum_{j=0}^\infty \frac{I_{A,C}(p^{j})
   }{p^{j(1-\hat \alpha)}}   }\\
   &=& \frac{p}{p-1}
\frac{\prod_{\alpha\in A} (1-\frac{1}{p^{1-\hat \alpha+\alpha}})}
{\prod_{\gamma\in C} (1-\frac{1}{p^{1-\hat\alpha+  \gamma}})}
  \sum_{j=0}^\infty \frac{I_{A',C}(p^{j+J})
   }{p^{j(1-\hat \alpha)}} \\
   &=&
   \frac{\prod_{\alpha\in A'} (1-\frac{1}{p^{1-\hat \alpha+\alpha}})}
{\prod_{\gamma\in C} (1-\frac{1}{p^{1-\hat\alpha+  \gamma}})}
  \sum_{j=0}^\infty \frac{I_{A',C}(p^{j+J})
   }{p^{j(1-\hat \alpha)}} .
 \end{eqnarray*}
 
 Now, all that is left to do is to prove that
  \begin{eqnarray*}
 \sum_{\ell=0}^\infty \frac{\mu(p^\ell)}{p^{\ell(2-\hat \alpha-\hat \beta)}}\sum_{J=0}^\infty
 \frac{1}{p^J}\sum_{j=0}^\infty \frac{I_{A',C}(p^{j+\ell+J})}{p^{j(1-\hat \alpha)}}
 \sum_{k=0}^\infty \frac{I_{B',D}(p^{k+\ell+J})}{p^{k(1-\hat \beta)}}
 =\left(1-\frac{1}{p^{1-\hat \alpha-\hat \beta}}\right)\sum_{\ell=0}^\infty
 \frac{I_{A'\cup\{-\hat \beta\},C}(p^\ell)I_{B'\cup\{-\alpha\},D}(p^\ell)}{p^\ell}
 \end{eqnarray*}
 
  Temporarily let
 $X=1/p$, 
  $Y=p^{-\hat \alpha}$, 
   $Z=p^{-\hat \beta}$, $a_j=I_{A,C}(p^j)$, $a_j'=I_{A',C}(p^j)$, 
   $\tilde{a}_j=I_{A'\cup\{-\hat\beta\},C}(p^j)$; 
   and $b_k=I_{B,D}(p^k)$, $b_k'=I_{B',D}(p^k)$, $\tilde{b}_k=I_{B'\cup \{-\hat \alpha\},D}(p^k)$.
  Then the desired identity follows from the theorem of the next section.
 
  \section{The identity}
 \begin{theorem} Suppose that $a', b', \tilde{a}, \tilde{b}$ are sequences such that
  \begin{eqnarray*}
 \sum_{J=0}^\ell Z^{J-\ell} a'_{J}=\tilde{a}_\ell \qquad \sum_{K=0}^\ell Y^{K-\ell}b'_{K}=\tilde{b}_\ell.
\end{eqnarray*}
Then 
   \begin{eqnarray*}&&
   \sum_{J=0}^\infty \sum_{\ell=0}^{\min(1,J)}(-1)^\ell \frac{X^{2\ell+J}}{Y^\ell Z^\ell}\sum_{j=0}^\infty a'_{j+\ell+J}\left(\frac X Y\right)^j \sum_{k=0}^\infty b'_{k+\ell+J}\left(\frac X Z\right)^k = \left(1-\frac X{YZ}\right)\sum_{\ell=0}^\infty \tilde{a}_\ell\tilde{b}_\ell X^\ell.
   \end{eqnarray*}
   \end{theorem}
   \begin{proof}
The left side may be written as 
 \begin{eqnarray*}&&
   \sum_{J=0}^\infty   X^{ J} \sum_{j=0}^\infty a'_{j+J}\left(\frac X Y\right)^j \sum_{k=0}^\infty b'_{k+J}\left(\frac X Z\right)^k\\
   && \qquad  -
   \sum_{J=0}^\infty  X^J \sum_{j=0}^\infty a'_{j+1+J}\left(\frac X Y\right)^{j+1} \sum_{k=0}^\infty b'_{k+1+J}\left(\frac X Z\right)^{k+1}
   \end{eqnarray*}
   This may be rewritten as
   \begin{eqnarray*}&&
   \sum_{J=0}^\infty X^J \bigg[ \left( \sum_{j=0}^\infty a'_{j+J}\left(\frac X Y\right)^j -
    \sum_{j=0}^\infty a'_{j+1+J}\left(\frac X Y\right)^{j+1}\right) 
     \sum_{k=0}^\infty b'_{k+J}\left(\frac X Z\right)^k\\
   && \qquad  + \sum_{j=0}^\infty a'_{j+1+J}\left(\frac X Y\right)^{j+1}\left(
\sum_{k=0}^\infty b'_{k+J}\left(\frac X Z\right)^k-
   \sum_{k=0}^\infty b'_{k+1+J}\left(\frac X Z\right)^{k+1}
\right) \bigg]
\end{eqnarray*}
which simplifies to
  \begin{eqnarray*}&&
   \sum_{J=0}^\infty X^J \bigg[ a'_J     \sum_{k=0}^\infty b'_{k+J}\left(\frac X Z\right)^k 
    + b'_J \sum_{j=0}^\infty a'_{j+1+J}\left(\frac X Y\right)^{j+1} \bigg]
\end{eqnarray*}
Now
 \begin{eqnarray*}&&
   \sum_{J=0}^\infty X^J a'_J     \sum_{k=0}^\infty b'_{k+J}\left(\frac X Z\right)^k 
   =\sum_{\ell=0}^\infty X^\ell b'_\ell \sum_{J=0}^\ell a'_J Z^{J-\ell}=
   \sum_{\ell=0}^\infty \tilde{a}_\ell b'_\ell X^\ell .
    \end{eqnarray*}
    And
     \begin{eqnarray*}&&
   \sum_{J=0}^\infty X^J b'_J  \sum_{j=0}^\infty a'_{j+1+J}\left(\frac X Y\right)^{j+1}  \\
   && \qquad 
   =
    \sum_{J=0}^\infty X^J b'_J   \sum_{j=0}^\infty a'_{j+J}\left(\frac X Y\right)^{j} 
    - \sum_{J=0}^\infty a'_Jb'_J X^J\\
    &&\qquad = \sum_{\ell=0}^\infty a'_\ell \tilde{b}_\ell X^\ell  
    - \sum_{\ell=0}^\infty a'_\ell b'_\ell X^\ell.
\end{eqnarray*}
Thus, the left side of the identity is
\begin{eqnarray*}
\sum_{\ell=0}^\infty \tilde{a}_\ell b'_\ell X^\ell
+ \sum_{\ell=0}^\infty a'_\ell \tilde{b}_\ell X^\ell
    - \sum_{\ell=0}^\infty a'_\ell b'_\ell X^\ell.
\end{eqnarray*}
But $a'_\ell=\tilde{a}_\ell-\frac{\tilde{a}_{\ell-1}}{Z}$ and  $b'_\ell=\tilde{b}_\ell-\frac{\tilde{b}_{\ell-1}}{Y}$ 
so that 
\begin{eqnarray*}
  \tilde{a}_\ell b'_\ell
+ a'_\ell \tilde{b}_\ell 
    - a'_\ell b'_\ell=\tilde{a}_\ell \tilde{b}_\ell -\frac{\tilde{a}_{\ell-1}\tilde{b}_{\ell-1}}{YZ}.
    \end{eqnarray*}
    The sum over $\ell$ of this expression times $X^\ell$ gives the right side of the identity.
    \end{proof}
  
   The reader may have noticed the similarity between this identity and the corresponding identity that formed the crux of [CK3].


\begin{thebibliography}{[CFKRS]}
  
\bibitem[BK1]{journal}
E. B.  Bogomolny and J. P. Keating.  Random matrix theory and the Riemann zeros I: 
three- and four-point correlations. Nonlinearity 8 (1995), 1115--1131.

\bibitem[BK2]{journal}
 E. B.  Bogomolny and J. P. Keating. Random matrix theory and the Riemann zeros II: 
$n$-point correlations. Nonlinearity 9 (1996), 911--935.

\bibitem[CFZ]{journal}
J. B. Conrey, D. W. Farmer and M. R. Zirnbauer. 
Autocorrelation of ratios of $L$-functions. 
Commun. Number Theory Phys. 2 (2008), 593--636.

\bibitem[CK1]{CKI}
J.B. Conrey and J.P. Keating.  Moments of zeta and correlations of divisor-sums: I. Phil. Trans. R. Soc. A 373 (2015), 20140313;
 arXiv:1506.06842  

\bibitem[CK2]{CKII}
J.B. Conrey and J.P. Keating.  Moments of zeta and correlations of divisor-sums: II. In Advances in the Theory of Numbers -- Proceedings of the Thirteenth Conference of the Canadian Number Theory Association, Fields Institute Communications (Editors: A. Alaca, S. Alaca \& K.S. Williams), 75--85 (2015, Springer);
 arXiv:1506.06843 

\bibitem[CK3]{CKIII}
J.B. Conrey and J.P. Keating.  Moments of zeta and correlations of divisor-sums: III. Indagationes Mathematicae 26 (2015), no. 5, 736--747;  
arXiv:1506.06844 

\bibitem[CK4]{CKIV}
J.B. Conrey and J.P. Keating. Moments of zeta and correlations of divisor-sums: IV. Res. Number Theory 2 (2016),  2:24;  
arXiv:1506.06844 

\bibitem[CK5]{CK} J. B. Conrey and J. P. Keating. 
Pair correlation and twin primes revisited. Proc. R. Soc. A 472 (2016), 20160548;
arXiv:1604.06124

 \bibitem[CSn]{CSn}
J. B.  Conrey and N. C. Snaith. 
Applications of the L-functions ratios conjectures. Proc. Lond. Math. Soc. (3) 94 (2007), no. 3, 594--646.

\bibitem[DFI]{journal}
 W. Duke,J. B. Friedlander, and H. Iwaniec.  A quadratic divisor problem. Invent. Math. 115 (1994), no. 2, 209--217. 

\bibitem[GG]{journal}
D. A.  Goldston and S. M.  Gonek.  Mean value theorems for long Dirichlet polynomials and tails of 
Dirichlet series. Acta Arith. 84 (1998), no. 2, 155--192.

\end{thebibliography}
 \end{document}